\documentclass[11pt]{amsart}

\usepackage{amsmath,amsthm,amssymb,latexsym}
\usepackage{hyperref}
\newcommand{\uhr}{\upharpoonright}

\DeclareMathOperator{\rng}{rng}

\newcommand{\sub}[1]{_{\textrm{\tiny{#1}}}}
\renewcommand{\leq}{\leqslant}
\renewcommand{\geq}{\geqslant}
\renewcommand{\nleq}{\nleqslant}

\newcommand{\tless}{<\sub{T}}
\newcommand{\tgreater}{>\sub{T}}

\newcommand{\tleq}{\leq \sub{T}}
\newcommand{\ntleq}{\nleq \sub{T}}

\newcommand{\mc}{\mathcal}

\newtheorem{thm}{Theorem}[section]
\newtheorem{cor}[thm]{Corollary}
\newtheorem{lem}[thm]{Lemma}
\newtheorem{prop}[thm]{Proposition}

\theoremstyle{definition}
\newtheorem{defn}[thm]{Definition}
\newtheorem{q}[thm]{Question}

\title[The Strength of Some Combinatorial Principles Related to
RT$^2_2$]{The Strength of Some Combinatorial Principles Related to
\\Ramsey's Theorem for Pairs}
\author[Hirschfeldt]{Denis R. Hirschfeldt}
  \address{Department of Mathematics\\
  The University of Chicago\\
  Chicago, IL 60637-1514, USA}
  \email{drh@math.uchicago.edu}
\author[Jockusch]{Carl G. Jockusch, Jr.}
  \address{Department of Mathematics\\
  1409 W. Green St.\\
  University of Illinois at Urbana-Cham\-paign\\
  Urbana, IL 61801-2975, USA}
  \email{jockusch@math.uiuc.edu}
\author[Kjos-Hanssen]{Bj\o rn Kjos-Hanssen}
  \address{Department of Mathematics\\
  University of Connecticut\\
  Storrs, CT 06269-3009, USA}
  \email{bjorn@math.uconn.edu}
\author[Lempp]{Steffen Lempp}
  \address{Department of Mathematics\\
  University of Wisconsin\\
  Madison, WI 53706-1388, USA}
  \email{lempp@math.wisc.edu}
\author[Slaman]{Theodore A. Slaman}
  \address{Department of Mathematics\\
  The University of California\\
  Berkeley, CA 94720-3840, USA}
  \email{slaman@math.berkeley.edu}
\thanks{All of the authors thank the Institute for Mathematical
Sciences of the National University of Singapore for its generous
support during the period July 18 -- August 15, 2005, when part of
this research was carried out. The first author was partially
supported by NSF grants DMS-0200465 and DMS-0500590. The second author
was supported by NSFC Grand International Joint Project `New
Directions in Theory and Applications of Models of Computation', No.\
60310213.  The fourth author was partially supported by NSF grant
DMS-0140120. The fifth author was partially supported by NSF grant
DMS-0501167.}

\begin{document}

\begin{abstract}
We study the reverse mathematics and computability-the\-o\-re\-tic
strength of (stable) Ramsey's Theorem for pairs and the related
principles COH and DNR. We show that SRT$^2_2$ implies DNR over
RCA$_0$ but COH does not, and answer a question of Mileti by showing
that every computable stable $2$-coloring of pairs has an incomplete
$\Delta^0_2$ infinite homogeneous set. We also give some extensions of
the latter result, and relate it to potential approaches to showing
that SRT$^2_2$ does not imply RT$^2_2$.
\end{abstract}

\maketitle

\section{Introduction}

\noindent In this paper we establish some results on the reverse
mathematics and computability-theoretic strength of combinatorial
principles related to Ramsey's Theorem for pairs. This topic has
attracted a large amount of recent research (see for instance
\cite{CJS,HS,Mileti,Milbull}), but certain basic questions still
remain open.

For a set $X$, let $[X]^2=\{Y \subset X \mid |Y|=2\}$. A
\emph{$2$-coloring} of $[\mathbb N]^2$ is a function from $[\mathbb
N]^2$ into $\{0,1\}$. A set $H \subseteq \mathbb N$ is
\emph{homogeneous} for a $2$-coloring $C$ of $[\mathbb N]^2$ if $C$ is
constant on $[H]^2$. Ramsey's Theorem for pairs (RT$^2_2$) is the
statement in the language of second-order arithmetic that every
$2$-coloring of $[\mathbb N]^2$ has an infinite homogeneous set. A
$2$-coloring $C$ of $[\mathbb N]^2$ is \emph{stable} if for each $x
\in \mathbb N$ there exists a $y \in \mathbb N$ and a $c<2$ such that
$C(\{x,z\})=c$ for all $z>y$. Stable Ramsey's Theorem for pairs
(SRT$^2_2$) is RT$^2_2$ restricted to stable colorings.

It follows from work of Jockusch \cite[Theorem 5.7]{Jockusch} that if
$n>2$ then Ramsey's Theorem for $n$-tuples is equivalent to
arithmetical comprehension (ACA$_0$), but Seetapun \cite{SS} showed
that RT$^2_2$ does not imply ACA$_0$. (All implications and
nonimplications discussed here are over the standard base theory
RCA$_0$ of reverse mathematics. For background on reverse mathematics
and discussions of many of the techniques used below, see Simpson
\cite{Si98}.)

A long-standing open question in reverse mathematics is whether
RT$^2_2$ implies Weak K\"onig's Lemma (WKL$_0$), the statement that
every computable infinite binary tree has an infinite path. (That
WKL$_0$ does not imply RT$^2_2$ follows from a result of Jockusch
\cite[Theorem 3.1]{Jockusch} discussed below.) As is well-known,
WKL$_0$ is equivalent to the statement that for each set $A$, there is
a $0,1$-valued function function $f$ that is diagonally noncomputable
relative to $A$ (where a total function $f$ is \emph{diagonally
noncomputable} if $\forall e\,(f(e) \neq \Phi_e(e))$.) A natural way
to weaken this statement is to drop the requirement that $f$ be
$0,1$-valued, and allow it to take arbitrary values in $\omega$; the
corresponding axiom system has been named DNR. In Section \ref{first}
we show that RT$^2_2$ implies DNR over RCA$_0$. In other words,
whereas we do not know whether RT$^2_2$ implies WKL$_0$, we have a
partial result toward this implication. In fact, we show that the
possibly weaker system SRT$^2_2$ already implies DNR. It is not known
whether SRT$^2_2$ is strictly weaker than RT$^2_2$; we will discuss
this question further below.

An infinite set $X$ is \emph{cohesive} for a family $R_0,R_1,\ldots$
of sets if for each $i$, one of $X \cap R_i$ or $X \cap
\overline{R}_i$ is finite. COH is the principle stating that every
family of sets has a cohesive set.  Having seen that $\mbox{RCA}_0 +
\mbox{SRT}^2_2 \vdash \mbox{DNR}$, and recalling that RT$^2_2$ is
equivalent over RCA$_0$ to SRT$^2_2$ + COH (see \cite[Lemma 7.11]{CJS}
and \cite[Corollary A.1.4]{Mileti}), we proceed to compare COH and
DNR. As noted by Cholak, Jockusch, and Slaman \cite[Lemma 9.14]{CJS},
even WKL$_0$ does not imply COH, so certainly DNR does not imply
COH. We establish that COH does not imply DNR in Section \ref{second}.
This result was independently and simultaneously obtained by
Hirschfeldt and Shore \cite[Corollary 2.21]{HS}, and as we will see,
the main ideas of the proof were already present in \cite{CJS}.

Jockusch~\cite[Theorem 3.1]{Jockusch} constructed a computable
$2$-coloring of $[\mathbb N]^2$ with no $\Delta^0_2$ infinite
homogeneous set. On the other hand, computable stable $2$-colorings
always have $\Delta^0_2$ infinite homogeneous sets. Indeed, the
problem of finding an infinite homogeneous set for a computable stable
$2$-coloring is essentially the same as the problem of finding an
infinite subset of either $A$ or $\overline{A}$ for a $\Delta^0_2$ set
$A$. More precisely, we have the following. If $A$ is $\Delta^0_2$
then there is a computable stable $2$-coloring $C$ of $[\mathbb N]^2$
such that if $H$ is homogeneous for $C$ then $H \subseteq A$ or $H
\subseteq \overline{A}$. Conversely, if $C$ is a computable stable
$2$-coloring of $[\mathbb N]^2$ then there is a $\Delta^0_2$ set $A$
such that any infinite set $B$ with $B \subseteq A$ or $B \subseteq
\overline{A}$ computes an infinite homogeneous set for $C$. (See
\cite[Proposition 2.1]{Jockusch} and \cite[Lemma 3.5]{CJS}, or
\cite[Claim 5.1.3]{Mileti}.)

Cholak, Jockusch, and Slaman~\cite[Theorem 3.1]{CJS} showed that every
computable $2$-col\-or\-ing of $[\mathbb N]^2$ has a low$_2$ infinite
homogeneous set, and suggested the possibility of separating SRT$^2_2$
and RT$^2_2$ by showing that every computable \emph{stable}
$2$-coloring of $[\mathbb N]^2$ has a \emph{low} infinite homogeneous
set. Such a result, if relativizable, would allow us to build an
$\omega$-model of SRT$^2_2$ consisting entirely of low sets, which
would therefore not be a model of RT$^2_2$. (An $\omega$-model of
second-order arithmetic is one whose first-order part is standard, and
such a model is identified with its second-order part.) However,
Downey, Hirschfeldt, Lempp, and Solomon~\cite{DHLS} constructed a
computable stable $2$-coloring of $[\mathbb N]^2$ with no low infinite
homogeneous set.

Mileti~\cite[Theorem 5.3.7]{Mileti} showed that for each $X \tless 0'$
there is a computable stable $2$-coloring of $[\mathbb N]^2$ with no
$X$-computable infinite homogeneous set. (He also showed that this is
true for any low$_2$ set $X$.)

In light of these results, Mileti~\cite[Question 5.3.8]{Mileti} asked
whether there is an infinite $\Delta^0_2$ set $A$ such that every
infinite $\Delta^0_2$ subset of $A$ or $\overline{A}$ is complete
(i.e., has degree $\bf 0'$); in other words, whether there is a
computable stable $2$-coloring of $[\mathbb N]^2$ such that any
$\Delta^0_2$ infinite homogeneous set is complete. Hirschfeldt gave a
negative answer to this question; this previously unpublished result
appears as Corollary \ref{original} below. In Theorem \ref{incthm}, we
modify the proof of this result to show that, in fact, if
$C_0,C_1,\ldots \tgreater 0$ are uniformly $\Delta^0_2$, then for
every $\Delta^0_2$ set $A$ there is a $\Delta^0_2$ subset $X$ of
either $A$ or $\overline{A}$ such that $\forall i\,(C_i \ntleq X)$. In
proving that RT$^2_2$ does not imply ACA$_0$, Seetapun \cite{SS}
showed that if $C_0,C_1,\ldots \tgreater 0$ then every $2$-coloring of
$[\mathbb N]^2$ has an infinite homogeneous set that does not compute
any of the $C_i$. Our result can be seen as a $\Delta^0_2$ analogue of
this theorem. The restriction to stable colorings is of course
necessary in this case, since as mentioned above, there are
$2$-colorings of pairs with no $\Delta^0_2$ infinite homogeneous set.

There is still a large gap between the negative answer to Mileti's
question and the result of Downey, Hirschfeldt, Lempp, and
Solomon~\cite{DHLS} mentioned above. In particular, we would like to
know the answer to the following question.

\begin{q} \label{low2q}
Let $A$ be $\Delta^0_2$. Must there be an infinite subset of either
$A$ or $\overline{A}$ that is both $\Delta^0_2$ and low$_2$?
\end{q}

A relativizable positive answer to this question would lead to a
separation between SRT$^2_2$ and RT$^2_2$, since it would allow us to
build an $\omega$-model of RCA$_0$ + SRT$^2_2$ that is not a model of
RT$^2_2$, as we now explain. We begin with the $\omega$-model
$\mathcal M_0$ consisting of the computable sets. Let $C_0$ be a
stable $2$-coloring of $[\mathbb N]^2$ in $\mathcal M_0$. Assuming a
positive answer to Question \ref{low2q}, we have an infinite
homogeneous set $H_0$ for $C_0$ that is both $\Delta^0_2$ and
low$_2$. Note that $H'_0$ is low over $0'$ and c.e.\ over $0'$.

Now let $\mathcal M_1$ be the $\omega$-model consisting of the
$H_0$-computable sets, and let $C_1$ be a stable $2$-coloring of
$[\mathbb N]^2$ in $\mathcal M_1$. Again assuming a (relativizable)
positive answer to Question \ref{low2q}, we have an infinite
homogeneous set $H_1$ for $C_1$ such that $H_0 \oplus H_1$ is both
$\Delta^0_2$ in $H_0$ and low$_2$. As before, $(H_0 \oplus H_1)'$ is
low over $0'$. It may no longer be c.e.\ over $0'$, but it is $2$-CEA
over $0'$ (that is, it is c.e.\ in and above a set that is itself
c.e.\ in and above $0'$).

Now let $\mathcal M_2$ be the $\omega$-model consisting of the $H_0
\oplus H_1$-computable sets, and continue in this way, making sure
that for every $i$ and every stable $2$-coloring $C$ of $[\mathbb
N]^2$ in $\mathcal M_i$, we have $C_j=C$ for some $j$. Let $\mathcal M
= \bigcup_i \mathcal M_i$. By construction, $\mathcal M$ is an
$\omega$-model of RCA$_0$ + SRT$^2_2$, and for every set $X$ in
$\mathcal M$, we have that $X'$ is low over $0'$ and $m$-CEA over $0'$
for some $m$. By the extension of Arslanov's Completeness Criterion
given by Jockusch, Lerman, Soare, and Solovay~\cite{JLSS}, no such $X$
can have PA degree over $0'$ (that is, $X$ cannot be the degree of a
nonstandard model of arithmetic with an extra predicate for
$0'$). However, Jockusch and Stephan \cite[Theorem 2.1]{JS} showed
that a degree contains a p-cohesive set (that is, a set that is
cohesive for the collection of primitive recursive sets) if and only
if its jump is PA over $0'$. Thus $\mathcal M$ is not a model of COH,
and hence not a model of RT$^2_2$.

Note that to achieve the separation described above, it would be
enough to show (in a relativizable way) that every $\Delta^0_2$ set
$A$ has a subset of either it or its complement that is both
$\Delta^0_2$ and low$_n$ for some $n$ (which may depend on
$A$). However, we do not even know whether every $\Delta^0_2$ set has
a subset of either it or its complement that is both $\Delta^0_2$ and
nonhigh.

The ultimate refutation of this approach to separating SRT$^2_2$ and
RT$^2_2$ would be to build a computable \emph{stable} $2$-coloring of
$[\mathbb N]^2$ for which the jump of every infinite homogeneous set
has PA degree over $0'$. (Without the condition of stability, such a
coloring was built by Cholak, Jockusch, and Slaman~\cite[Theorem
12.5]{CJS}.) Indeed, such a construction (if relativizable) would show
that every $\omega$-model of RCA$_0$ + SRT$^2_2$ is a model of
RT$^2_2$, as we now explain. Suppose that such stable colorings exist,
and let $\mathcal M$ be an $\omega$-model of RCA$_0$ +
SRT$^2_2$. Relativizing the result of Jockusch and Stephan
\cite[Theorem 2.1]{JS} on p-cohesive sets mentioned above, we can show
that $\mathcal M$ is a model of COH. But as mentioned above, SRT$^2_2$
+ COH is equivalent to RT$^2_2$ over RCA$_0$, so $\mathcal M$ is a
model of RT$^2_2$.

\section{SRT$^2_2$ implies DNR} \label{first}

\noindent The proof that SRT$^2_2$ implies DNR over RCA$_0$ is
naturally given in two parts: first we show that each $\omega$-model
of SRT$^2_2$ is a model of DNR, and then that we can in fact carry out
the proof of this implication in RCA$_0$, that is, using only
$\Sigma^0_1$-induction.

\subsection{The argument for $\omega$-models}

A set $A$ is \emph{effectively bi-immune} if there is a computable
function $f$ such that for each $e$, if $W_e \subseteq A$ or $W_e
\subseteq \overline{A}$, then $|W_e|<f(e)$.

\begin{lem}\label{carlaugust2}
There is an effectively bi-immune set $A \tleq 0'$. In fact, we
can choose the function $f$ witnessing the bi-immunity of $A$ to be
defined by $f(e) = 3e+2$.
\end{lem}

\begin{proof}
We build $A$ in stages, via a $0'$-computable construction. At each
stage we decide the value of $A(n)$ for at most three $n$'s. At stage
$e$, we check whether $W_e$ has at least $3e+2$ many elements. If
so, then there are at least two elements $n_0,n_1 \in W_e$ at which
we have not yet decided the value of $A$. Let $A(n_0)=0$ and
$A(n_1)=1$. In any case, if $A(e)$ is still undefined then let
$A(e)=0$.
\end{proof}

We also need the following lemma, which follows immediately from the
equivalence mentioned above between finding homogeneous sets for
computable stable colorings and finding subsets of $\Delta^0_2$ sets
or their complements. A \emph{Turing ideal} is a subset of $2^\omega$
closed under Turing reduction and join. A subset of $2^\omega$ is a
Turing ideal if and only if it is an $\omega$-model of RCA$_0$.

\begin{lem} \label{ja}
A Turing ideal $\mc I$ is an $\omega$-model of SRT$^2_2$ if and only
if for each set $A$, if $A \tleq C'$ for some $C \in \mc I$, then
there is an infinite $B \in \mc I$ such that either $B \subseteq A$ or
$B \subseteq \overline A$.
\end{lem}

We can now prove the implication between SRT$^2_2$ and DNR for
$\omega$-models.

\begin{thm} \label{SRTDNR}
Each $\omega$-model of SRT$^2_2$ is a model of DNR.
\end{thm}

\begin{proof}
Let $\mc I$ be a Turing ideal that is an $\omega$-model of
SRT$^2_2$. We show that $\mc I$ contains a diagonally noncomputable
function. The proof clearly relativizes to get a function that is
diagonally noncomputable relative to $X$ for any $X \in \mc I$.

Let $A$ be as in Lemma \ref{carlaugust2}. By Lemma \ref{ja}, there is
an infinite $B \in \mc I$ such that $B$ is a subset of $A$ or
$\overline A$. By the choice of $A$, for all $e$, if $W_e \subseteq B$
then $|W_e|<3e+2$.

Let $g$ be such that $W_{g(e)}$ is the set consisting of the first
$3e+2$ many elements of $B$ (in the usual ordering of $\omega$). For
any $e$, if $W_e=W_{g(e)}$ then $W_e \subseteq B$, and so $|W_e| <
3e+2$. But $|W_{g(e)}| = 3e+2$, so this is a contradiction. Thus
$\forall e\,(W_e \neq W_{g(e)})$.

Now let $f$ be a computable function such that
$W_{f(e)}=W_{\Phi_e(e)}$ if $\Phi_e(e)\!\downarrow$, and
$W_{f(e)}=\emptyset$ otherwise. Then $h = g \circ f$ is diagonally
noncomputable, since it is total and for each $e$, if
$\Phi_e(e)\!\downarrow$ then $W_{h(e)} \neq W_{f(e)} =
W_{\Phi_e(e)}$. But $h$ is also computable in $B$, and hence belongs
to $\mc I$.
\end{proof}

\subsection{The proof-theoretic argument}

We now simply need to analyze the above proof to ensure that
$\Sigma^0_1$-induction suffices to carry it out. The formal analog of
Lemma \ref{ja} is the statement that SRT$^2_2$ is equivalent to the
following principle, called D$^2_2$: For every $0,1$-valued function
$d(x,s)$, if $\lim_s d(x,s)$ exists for all $x$, then there is an
infinite set $B$ and a $j<2$ such that $\lim_s d(x,s) = j$ for all $x
\in B$. The equivalence of SRT$^2_2$ and D$^2_2$ over RCA$_0$ is
claimed in \cite[Lemma 7.10]{CJS}. However, the argument indicated
there for the $\mbox{D}^2_2 \rightarrow \mbox{SRT}^2_2$ direction
appears to require $\Pi^0_1$-bounding, which is not provable in
RCA$_0$. It is unknown whether $\mbox{D}^2_2 \rightarrow
\mbox{SRT}^2_2$ is provable in RCA$_0$.  Fortunately, we need only the
other direction, since we are starting with the assumption that
SRT$^2_2$ holds. This direction is proved as in \cite[Lemma
7.10]{CJS}, and we reproduce the proof here for the reader's
convenience. Work in $\mbox{RCA}_0 + \mbox{SRT}^2_2$. Let a function
$d(x,s)$ be given that satisfies the hypothesis of D$^2_2$. Give the
pair $\{x,s\}$ with $x < s$ the color $d(x,s)$.  The infinite
homogeneous set produced by SRT$^2_2$ for this stable coloring
satisfies the conclusion of D$^2_2$.

\begin{thm}
$\mbox{RCA}_0 \vdash \mbox{SRT}^2_2 \rightarrow \mbox{DNR}$.
\end{thm}

\begin{proof}
Given the existence of a set $A$ as in Lemma \ref{carlaugust2} (or
more precisely, of a function $d(x,s)$ such that $A(x) = \lim_s
d(x,s)$), the definition of the diagonally noncomputable function $h$
given in the proof of Theorem \ref{SRTDNR} can clearly be carried out
using D$^2_2$ and $\Sigma^0_1$-induction.

So the only part of the proof of Theorem \ref{SRTDNR} we need to
consider more carefully is the construction of $A$ and the
satisfaction of all bi-immunity requirements. More precisely, fix a
model $\mathcal M$ of RCA$_0$ + SRT$^2_2$. Within that model, we have
an enumeration of the $\mathcal M$-c.e.\ sets $W_0,W_1,\ldots$ (where
the indices range over all elements of the first-order part of
$\mathcal M$). We need to show the existence of a function $d(x,s)$ in
$\mathcal M$ such that $\lim_s d(x,s)$ exists for all $x$, and for
every $W_e$, if there is a $j<2$ such that $\forall x \in W_e\,
(\lim_s d(x,s) = j)$, then $|W_e| < 3e+2$. (We will actually be able
to use $2e+2$ instead of $3e+2$.)

We can build $d$ in much the same way as we built $A$, but we need to
be more careful because we no longer have access to an oracle for
$0'$. So we need a computable construction to replace the
$0'$-computable construction in the proof of Lemma \ref{ja}. Let
$R_e$ be the $e$th bi-immunity requirement.

In this construction, $R_e$ may control up to two numbers $n^0_e$ and
$n^1_e$ at any point in the construction. At stage $t=\langle e,s
\rangle$, if $|W_{e,s}| \geq 2e+2$, then for each $i<2$ such that
$n^i_e$ is undefined, define $n^i_e$ to be different from each
$n^j_{e'}$ for $e' \leq e$, and undefine all $n^j_{e'}$ for $e' >
e$. In any case, for each $n$, if $n=n^j_k$ for some $j$ and $k$, then
let $d(n,t)=j$, and otherwise let $d(n,t)=0$.

It is now easy to check (in RCA$_0$) that $\lim_t d(n,t)$ exists for
all $n$, since for each $n$, either $n$ is never controlled by a
requirement, in which case $d(n,t)=0$ for all $t$, or there is a stage
$t$ at which $n$ is controlled by $R_e$ for some $e$. In the latter
case, since control of a number can only pass to stronger
requirements, there are at most $e$ many $u \geq t$ such that
$d(n,u+1) \neq d(n,u)$.

The last thing we need to check is that each $R_e$ is satisfied. It
follows by induction that for each $e$, there are at most $2e$ many
numbers that are ever controlled by any $R_{e'}$ with $e'<e$, and thus
there is a stage $v_e$ by which all such numbers have been controlled
by such requirements. (This is an instance of $\Pi^0_1$-induction,
which holds in RCA$_0$ (see Simpson \cite[Lemma 3.10]{Si98}), using a
formula saying that for all finite sequences of size $2e+1$ of
distinct elements and for all $t$, it is not the case that each
element of the sequence has been controlled by some $R_{e'}$ with
$e'<e$ by stage $t$.) So if $|W_e| \geq 2e+2$, then picking a stage
$t=\langle e,s \rangle \geq v_e$ such that $|W_{e,s}| \geq 2e+2$, the
$n^i_e$ must be defined at stage $t$, and will never be undefined at a
later stage, so $\lim_u d(n^i_e,u) = i$. Thus $R_e$ is satisfied.
\end{proof}

\section{COH does not imply DNR} \label{second}

\noindent In this section we show that COH does not imply DNR over
RCA$_0$. We first recall a connection between diagonally
noncomputable functions and special $\Pi^0_1$ classes.

\begin{defn}
For $n \geq 1$ and $A \in 2^\omega$, a $\Pi^0_n$ subclass of
$2^\omega$ is \emph{$A$-special} if it has no $A$-computable
members. A class is \emph{special} if it is $\emptyset$-special.
\end{defn}

\begin{thm}[Jockusch and Soare {\cite[Corollary 1.3]{Pacific}}]
\label{2becomes1}
If $A$ computes an element of a special $\Pi^0_2$ class, then $A$
computes an element of a special $\Pi^0_1$ class.
\end{thm}

\begin{cor}\label{ok}
Any diagonally noncomputable function computes an element of a special
$\Pi^0_1$ class.
\end{cor}

\begin{proof}
Consider the special $\Pi^0_2$ class
\begin{multline*}
\{A \mid \forall x,t\,\exists y\,\exists s>t\, [\langle x,y\rangle\in
A \, \wedge\, \neg(\Phi_{x,s}(x)\!\downarrow=y)]\, \wedge \\
\forall x,a,b\, [(\langle x,a\rangle \in A\, \wedge\, \langle
x,b\rangle \in A) \rightarrow a=b]\}.
\end{multline*}
It is easy to check that any diagonally noncomputable function
computes an element of this class. The corollary now follows from
Theorem \ref{2becomes1}.
\end{proof}

We now consider the relationship between cohesiveness and special
$\Pi^0_1$ classes.

\begin{lem}[Cholak, Jockusch and Slaman {\cite[Lemma 9.16]{CJS}}]
\label{mathias}
Let $A \in 2^\omega$, let $P$ be an $A$-special $\Pi^0_1$ class, and
let $R_0,R_1,\ldots \tleq A$. Then there is an $\vec R$-cohesive set
$G$ that does not compute any element of $P$.
\end{lem}

This lemma is proved using Mathias forcing with $A$-computable
conditions. We will use two results about Mathias forcing, but since
we will not work with this notion directly, we refer to \cite[Section
9]{CJS}, \cite[Section 6]{BKLS}, and \cite[Section 2]{HS} for the
relevant definitions. Analyzing the proof of Lemma \ref{mathias}, we
immediately obtain the following result.

\begin{cor}[to the proof of Lemma \ref{mathias}] \label{222}
There is an $m\in\omega$ such that if $G$ is $m$-$A$-generic for
Mathias forcing with $A$-computable conditions, then $G$ is cohesive
with respect to any collection of sets $\vec{R} \tleq A$.
\end{cor}

It is clear that Lemma \ref{mathias} generalizes to deal with all
$\Pi^0_1$ classes at once; this is proved directly in \cite[Lemma
6.3]{BKLS}.

\begin{lem}[Binns, Kjos-Hanssen, Lerman, and Solomon {\cite[Lemma
6.3]{BKLS}}] \label{111}
Let $P$ be a $\Pi^0_1$ class and let $A$ be a set. Let $G$ be
$3$-$A$-generic for Mathias forcing with $A$-computable conditions. If
$P$ is $A$-special, then $P$ is $(G \oplus A)$-special.
\end{lem}

We are now ready to establish the result in the section heading.

\begin{thm}
There is an $\omega$-model of $\mbox{RCA}_0 + \mbox{COH}$ that is not
a model of DNR.
\end{thm}

\begin{proof}
Let $m \geq 3$ be as in Corollary \ref{222}. Let $A_0=\emptyset$, and
inductively let $A_{n+1}$ be $A_n \oplus G_n$, where $G_n$ is
$m$-$A_n$-generic for Mathias forcing with $A_n$-computable
conditions. Let $\mathcal I$ be the Turing ideal generated by
$\{A_n \mid n\in\omega\}$.

Let $\mathcal M$ be the $\omega$-model determined by $\mathcal I$. If
$\vec{R} \in \mathcal I$ is a collection of sets then $\vec{R} \tleq
A_n$ for some $n$. By Corollary \ref{222}, $G_n$ is
$\vec{R}$-cohesive.  Since $G_n \in \mathcal I$, it follows that
$\mathcal M$ is a model of COH.

On the other hand, if $B$ computes a diagonally noncomputable
function, then by Corollary \ref{ok}, there is a special $\Pi^0_1$
class $P$ such that $B$ computes an element of $P$. In other words,
$P$ is not $B$-special. However, if $B \in \mathcal I$ then $B \tleq
A_n$ for some $n$. By Lemma \ref{111} and induction, $P$ is
$A_n$-special, and hence $P$ is $B$-special. So if $B$ computes a
diagonally noncomputable function, then $B \notin \mathcal I$. Thus
$\mathcal M$ is not a model of DNR.
\end{proof}

So DNR separates SRT$^2_2$ from COH. That is, SRT$^2_2$ implies DNR,
whereas COH does not.

\section{Degrees of homogeneous sets for stable colorings} \label{last}

\noindent In this section we give our negative answer to Mileti's
question mentioned in the introduction. We will need two auxiliary
results. One is an extension of the low basis theorem noted by Linda
Lawton (unpublished).

\begin{thm}[Lawton] \label{lbt}
Let $T$ be an infinite, computable, computably bounded tree, and let
$C_0,C_1,\ldots \tgreater 0$ be uniformly $\Delta^0_2$. Then $T$ has
an infinite low path $P$ such that $\forall i\,(C_i \ntleq P)$, and an
index of such a $P$ can be $0'$-computed from an index of $T$.
\end{thm}

\noindent This theorem is proved by forcing with $\Pi^0_1$ classes,
and lowness is achieved just as in the usual proof of the low basis
theorem. Steps are interspersed to guarantee cone avoidance, which is
possible by the following lemma.

\begin{lem}
Let $C$ be a noncomputable set and let $Q$ be a nonempty computably
bounded $\Pi^0_1$ class. Let $\Phi$ be a Turing reduction.  Then $Q$
has a nonempty $\Pi^0_1$ subclass $R$ such that $\Phi^f \neq C$ for
all $f \in R$. Furthermore, there is a fixed procedure that computes
an index of $R$ from indices of $Q$ and $\Phi$ and an oracle for $C
\oplus 0'$.
\end{lem}

\begin{proof}
Let $U$ be a computable tree with $Q = [U]$. For each $n$, let $U_n$
be the set of strings $\sigma$ in $U$ such that $\Phi^\sigma(n)$ is
either undefined or has a value other than $C(n)$.  (Here we use the
convention that computations with string oracles $\sigma$ run for at
most $|\sigma|$ steps.) Then $U_n$ is a computable tree, and an index
of it can be computed from a $C$-oracle. Note that $U_n$ is infinite
for some $n$, since otherwise $C$ is computable.  Furthermore, $\{n
\mid U_n \mbox{ is infinite }\} \tleq C \oplus 0'$, since $C$ can
compute an index of $U_n$ as a computable tree, and then $0'$ can
determine whether $U_n$ is infinite by asking whether it contains a
string of every length. Let $R = [U_n]$ for the least $n$ with $U_n$
infinite.
\end{proof}

Below, we will use the following relativized form of Theorem
\ref{lbt}, which can be proved in the same way: Let $L$ be a low
set. Let $T$ be an infinite, $L$-computable, $L$-computably bounded
tree, and let $C_0,C_1,\ldots \ntleq L$ be uniformly
$\Delta^0_2$. Then $T$ has an infinite low path $P$ such that $\forall
i\,(C_i \ntleq P)$, and an index of such a $P$ can be $0'$-computed
from indices of $L$ and $T$.

The other result we will use below is that if $C_0,C_1,\ldots
\tgreater 0$ are uniformly $\Delta^0_2$ and the complement
$\overline{A}$ of the $\Delta^0_2$ set $A$ has no infinite
$\Delta^0_2$ subset $Y$ such that $\forall i\,(C_i \ntleq Y)$, then
$A$ cannot be too sparse.

\begin{defn}
An infinite set $Z$ is \emph{hyperimmune} if for every computable
increasing function $f$, there is an $n$ such that the interval
$[f(n),f(n+1))$ contains no element of $Z$.

If $Z$ is not hyperimmune, then a computable $f$ such that
$[f(n),f(n+1)) \cap Z \neq \emptyset$ is said to \emph{witness the
non-hyperimmunity of $Z$}.
\end{defn}

\begin{prop} \label{auxprop}
Let $A$ be $\Delta^0_2$. Let $C_0,C_1,\ldots$ be uniformly
$\Delta^0_2$ and let $L$ be an infinite $\Delta^0_2$ set such that
$C_i \ntleq L$ for all $i$. If $A \cap L$ is $L$-hyperimmune, then
there is an infinite $\Delta^0_2$ set $Y \subseteq \overline{A}$ such
that $\forall i\,(C_i \ntleq Y)$.
\end{prop}

\begin{proof}
We build $Y$ by finite extensions; that is, we define $\gamma_0 \prec
\gamma_1 \prec \cdots$ and let $Y = \bigcup_i \gamma_i$.

For a string $\sigma$ and a set $X$, we write $\sigma \sqsubset X$ to
mean that $\{n < |\sigma| \mid \sigma(n) = 1\} \subseteq X$.

Begin with $\gamma_0$ defined as the empty sequence. At stage
$s=\langle e,i \rangle$, given the finite binary sequence $\gamma_s
\sqsubset \overline{A} \cap L$, we $0'$-computably search for either
\begin{enumerate}

\item an $m$ and extensions $\gamma_s\sigma_0$ and $\gamma_s\sigma_1$
such that $\Phi_e^{\gamma_s\sigma_0}(m)\!\downarrow \neq
\Phi_e^{\gamma_s\sigma_1}(m)\!\downarrow$ and $\gamma_s\sigma_k
\sqsubset \overline{A} \cap L$ for $k=0,1$; or

\item an $m$ such that for all extensions $\gamma_s0^m\sigma \sqsubset
L$, either $\Phi_e^{\gamma_s0^m\sigma}(m)\!\uparrow$ or
$\Phi_e^{\gamma_s0^m\sigma}(m)\!\downarrow \neq C_i(m)$.

\end{enumerate}

We claim one of these must be found. Suppose not. Then for every $m$
we can find an extension $\gamma_s0^m\sigma_0 \sqsubset L$ such that
$\Phi_e^{\gamma_s0^m\sigma_0}(m)\!\downarrow = C_i(m)$. Since $C_i
\ntleq L$, there must be infinitely many $m$ for which there is also
an extension $\gamma_s0^m\sigma_1 \sqsubset L$ such that
$\Phi_e^{\gamma_s0^m\sigma_1}(m)\!\downarrow \neq C_i(m)$. So we can
$L$-computably enumerate an infinite set $M$ such that for each $m \in
M$, there are $\gamma_s0^m\sigma_k \sqsubset L$ for $k=0,1$ such that
$\Phi_e(\gamma_s0^m\sigma_0)\!\downarrow \neq
\Phi_e(\gamma_s0^m\sigma_1)\!\downarrow$. Let $m \in M$. Since we are
assuming that case 1 above does not hold, there must be a $k$ such
that $\gamma_s0^m\sigma_k \not\sqsubset \overline{A} \cap L$. So
letting $l_m$ be the maximum of $|\gamma_s0^m\sigma_k|$ for $k=0,1$,
we are guaranteed the existence of an element of $A \cap L$ in the
interval $[m,l_m)$. Now we can find $m_0,m_1,\ldots \in M$ such that
$m_{j+1} > l_{m_j}$, and define $f(j)=m_j$. Then $f$ is a witness to
the non-$L$-hyperimmunity of $A \cap L$, contrary to hypothesis.

So one of the two cases above must eventually hold. If case 1 holds,
let $k$ be such that $\Phi_e^{\gamma_s\sigma_k}(m) \neq C_i(m)$ and
define $\gamma'_s = \gamma_s\sigma_k$. If case 2 holds, define
$\gamma'_s= \gamma_s0^m$. In either case, let $\gamma_{s+1} \sqsubset
\overline{A} \cap L$ be an extension of $\gamma'_s$ such that
$\gamma_{s+1}(j)=1$ for some $j>|\gamma_s|$. Such a string must exist
since $\gamma'_s \sqsubset \overline{A} \cap L$ and $\overline{A} \cap
L$ is infinite (as otherwise $A \cap L$ would be cofinite within $L$,
and hence not $L$-hyperimmune). This definition ensures that $\Phi_e^Y
\neq C_i$.
\end{proof}

We are now ready to prove the main result of this section.

\begin{thm} \label{incthm}
Let $A$ be $\Delta^0_2$ and let $C_0,C_1,\ldots \tgreater 0$ be
uniformly $\Delta^0_2$. Then either $A$ or $\overline{A}$ has an
infinite $\Delta^0_2$ subset $X$ such that $C_i \ntleq X$ for all $i$.
\end{thm}

\begin{proof}
Assume that $\overline{A}$ has no infinite $\Delta^0_2$ subset $Y$
such that $C_i \ntleq Y$ for all $i$. We use Proposition
\ref{auxprop} to build an infinite $\Delta^0_2$ set $X$ such that $C_i
\ntleq X$ for all $i$, via a $0'$-computable construction satisfying
the following requirements:
$$
R_{e,i} : \Phi_e^X \mbox{ total} \;\; \Rightarrow\; \exists
n\,(\Phi_e^X(n) \neq C_i(n)).
$$

We first discuss how to satisfy the single requirement $R_{0,0}$. By
Proposition \ref{auxprop} (with $L=\omega$), $A$ is not
hyperimmune. Suppose we have a computable function $f$ witnessing the
non-hyperimmunity of $A$. Let the computable, computably bounded tree
$\widehat{T}$ consist of the nodes $(m_0,\ldots,m_{k-1})$ with $f(j)
\leq m_j < f(j+1)$ for all $j<k$. Such a node represents a guess that
$m_j \in A$ for each $j<k$. Note that the choice of $f$ ensures that
$\widehat{T}$ has at least one path along which all such guesses are
correct.

Now prune $\widehat{T}$ as follows. For each node $\sigma =
(m_0,\ldots,m_{k-1})$, if there are nonempty $F_0,F_1 \subseteq
\rng(\sigma)$ and an $n$ such that $\Phi^{F_0}_0(n)\!\downarrow \neq
\Phi^{F_1}_0(n)\!\downarrow$ with uses bounded by the largest element
of $F_0 \cup F_1$, then prune $\widehat{T}$ to ensure that $\sigma$ is
not extendible to an infinite path. Note that we can do this pruning
in such a way as to end up with a computable tree $T$.

Now $0'$ can determine whether $T$ is finite. If so, then we can find
a leaf $\sigma$ of $T$ such that $\rng(\sigma) \subset A$. There are
nonempty $F_0,F_1 \subseteq \rng(\sigma)$ and an $n$ such that
$\Phi^{F_0}_0(n)\!\downarrow \neq \Phi^{F_1}_0(n)\!\downarrow$ with
uses bounded by the largest element $z$ of $F_0 \cup F_1$, so if we
let $k$ be such that $\Phi^{F_k}_0(n) \neq C_0(n)$ and define $X$ so
that $X \uhr z+1 = F_k \uhr z+1$, then we ensure that $\Phi_0^X(n)
\neq C_0(n)$.

On the other hand, if $T$ is infinite then by Theorem \ref{lbt}, $0'$
can find a low path $P$ of $T$ such that $C_i \ntleq P$ for all
$i$. There must be an $n$ such that either $\Phi_0^Y(n)\!\uparrow$ for
every $Y \subseteq \rng(P)$ or there is a $Y \subseteq \rng(P)$ such
that $\Phi_0^Y(n)\!\downarrow \neq C_0(n)$, since otherwise we could
$P$-compute $C_0(n)$ for each $n$ by searching for a finite $F \subset
\rng(P)$ such that $\Phi_0^F(n)\!\downarrow$. But by the construction
of $T$, this means that there is an $n$ such that for every infinite
$Y \subseteq \rng(P)$, either $\Phi_0^Y(n)\!\uparrow$ or
$\Phi_0^Y(n)\!\downarrow \neq C_0(n)$. So if we now promise to make $X
\subseteq \rng(P)$, we ensure that $\Phi_0^X \neq C_0$. Notice that we
can make such a promise because $C_i \ntleq P$ for all $i$, and hence
$C_i \ntleq \rng(P) $ for all $i$ (since $P$ is an increasing
sequence), which implies that $A \cap \rng(P)$ is infinite.

Let us now consider how to satisfy another requirement, say
$R_{0,1}$. The action taken to satisfy $R_{0,0}$ results in either a
finite initial segment of $X$ being determined, or a promise being
made to keep $X$ within a given infinite low set that does not compute
any of the $C_i$. We can handle both cases at once by assuming that we
have a number $r_1$ and an infinite low set $L_1$ containing the
finite set $F_1$ of numbers less than $r_1$ currently in $X$, such
that $C_i \ntleq L_1$ for all $i$. We want $X \uhr r_1 = F_1$ and $X
\subseteq L_1$.

Suppose that we have an $L_1$-computable function $g$ witnessing the
non-$L_1$-hyperimmunity of $A \cap L_1$. We can then proceed much as
we did for $R_{0,0}$, but taking $r_1$ and $L_1$ into account, in the
following way. We can assume that $g(0) \geq r_1$. Define
$\widehat{T}$ to consist of the nodes $(m_0,\ldots,m_{k-1})$ with
$g(j) \leq m_j < g(j+1)$ and $m_j \in L_1$ for all $j<k$. For each
node $\sigma = (m_0,\ldots,m_{k-1})$, if there are nonempty $G_0,G_1
\subseteq \rng(\sigma)$ and an $n$ such that $\Phi^{F_1 \cup
G_0}_1(n)\!\downarrow \neq \Phi^{F_1 \cup G_1}_1(n)\!\downarrow$ with
uses bounded by the largest element of $G_0 \cup G_1$, then prune
$\widehat{T}$ to ensure that $\sigma$ is not extendible to an infinite
path, thus obtaining a new $L_1$-computable tree $T$.

If $T$ is finite then find a leaf $\sigma$ of $T$ such that
$\rng(\sigma) \subset A \cap L_1$. Then there is a nonempty $G
\subseteq \rng(\sigma)$ and an $n$ such that $\Phi^{F_1 \cup
G}_1(n)\!\downarrow \neq C_0(n)$ with use bounded by the largest
element $z$ of $G$, so if we define $X$ such that $X \uhr z+1 = (F_1
\cup G) \uhr z+1$ then we ensure that $\Phi_1^X(n) \neq C_0(n)$.

If $T$ is infinite then $0'$ can find a low path $P$ of $T$. If we now
promise that all future elements of $X$ will be in $\rng(P)$, we
ensure that $\Phi_1^X \neq C_0$ as before. Notice that we can make
such a promise because $\rng(P) \subseteq L_1$ and, as before, $A \cap
\rng(P)$ is infinite.

Thus we can satisfy $R_{0,1}$, and the action we take results in a
number $r_2$ and an infinite low set $L_2$ that does not compute any
of the $C_i$ (and contains the finite set $F_2$ of numbers less than
$r_2$ currently in $X$) such that we want $X \uhr r_2 = F_2$ and $X
\subseteq L_2$. In other words, we are in the same situation we were
in after satisfying $R_{0,0}$, and we could now proceed to satisfy
another requirement as we did $R_{0,1}$.

However, there is a crucial problem with proceeding in this way for
all the $R_{e,i}$ at once, which is that we know no $0'$-computable
way to determine the witnesses to non-hyperimmunity required by the
construction. The best we can do is guess at them. That is, we have a
$0''$-partial computable function $w$ such that if $l$ is a lowness
index for an infinite set $L$ (that is, $\Phi_l^{0'} = L'$) then
$\Phi^L_{w(l)}$ witnesses the non-$L$-hyperimmunity of $A \cap L$.
\smallskip

We are now ready to describe our construction. We give our
requirements a priority ordering by saying that $R_{e,i}$ is stronger
than $R_{e',i'}$ if $\langle e,i \rangle < \langle e',i' \rangle$. All
numbers added to $X$ at a stage $s$ of our construction will be
greater than $s$, thus ensuring that $X \tleq \emptyset'$. Let $X_s$
be the set of numbers added to $X$ by the beginning of stage $s$.

Throughout the construction, we run a $0'$-approximation to
$w$. Associated with each $R_{e,i}$ are a number $r_{\langle e,i
\rangle}$ and a low set $L_{\langle e,i \rangle}$ with lowness index
$l_{\langle e,i \rangle}$ (all of which might change during the
construction). If the approximation to $w(l_{\langle e,i \rangle})$
changes, then for all $\langle e',i' \rangle \geq \langle e,i \rangle$
the strategy for $R_{e',i'}$ is immediately canceled, $R_{e',i'}$ is
declared to be unsatisfied, and $r_{\langle e',i' \rangle}$,
$L_{\langle e',i' \rangle}$, and $l_{\langle e',i' \rangle}$ are reset
to the current values of $r_{\langle e,i \rangle}$, $L_{\langle e,i
\rangle}$, and $l_{\langle e,i \rangle}$, respectively. It is
important to note that the approximation to $w$ continues to run
during the action of a strategy at a fixed stage. That is, we may find
a change in the approximation to some $w(l_{\langle e,i \rangle})$
with $\langle e,i \rangle \leq \langle e',i' \rangle$ in the middle of
a stage $s$ at which we are trying to satisfy $R_{e',i'}$. If this
happens then we immediately end the stage and cancel strategies as
described above.

Initially, all requirements are unsatisfied. At the beginning of stage
$0$, for every $e,i$, let $r_{\langle e,i \rangle}=0$ and $L_{\langle
e,i \rangle}=\omega$, and let $l_{\langle e,i \rangle}$ be a fixed
lowness index for $\omega$.

At stage $s$, let $R_{e,i}$ be the strongest unsatisfied requirement
and proceed as follows.

We have a number $r_{\langle e,i \rangle}$ and a low set $L_{\langle
e,i \rangle}$ with lowness index $l_{\langle e,i \rangle}$, such that
$L_{\langle e,i \rangle}$ contains $X_s \uhr r_{\langle e,i \rangle}$,
and $C_j \ntleq L_{\langle e,i \rangle}$ for all $j$. As before, we
want to ensure that $X \uhr r_{\langle e,i \rangle} = X_s \uhr
r_{\langle e,i \rangle}$ and $X \subseteq L_{\langle e,i
\rangle}$. Let $v$ be the current approximation to $w(l_i)$ and let
$g=\Phi^{L_{\langle e,i \rangle}}_v$. By shifting the values of $g$ if
necessary, we can assume that $g(0) \geq \max(r_{\langle e,i
\rangle},s)$. Define $\widehat{T}$ to consist of the nodes
$(m_0,\ldots,m_{k-1})$ with $g(j) \leq m_j < g(j+1)$ and $m_j \in
L_{\langle e,i \rangle}$ for all $j<k$. Note that $g$ may not be
total, in which case $\widehat{T}$ is finite.

For each node $\sigma = (m_0,\ldots,m_{k-1})$, if there are nonempty
$G_0,G_1 \subseteq \sigma$ and an $n$ such that $\Phi^{X \uhr r_i \cup
G_0}_e(n)\!\downarrow \neq \Phi^{X \uhr r_i \cup
G_1}_e(n)\!\downarrow$ with uses bounded by the largest element of
$G_0 \cup G_1$, then prune $\widehat{T}$ to ensure that $\sigma$ is
not extendible to an infinite path, thus obtaining a new $L_{\langle
e,i \rangle}$-computable tree $T$.

We want to $0'$-effectively determine whether $T$ is finite. More
precisely, the question we ask $0'$ is whether the pruning process
described above ever results in all the nodes at some level of
$\widehat{T}$ becoming non-extendible. A positive answer means $T$ is
finite. If $g$ is total then a negative answer means $T$ is
infinite. However, if $g$ is not total, so that $\widehat{T}$ is
finite, we may still get a negative answer, because the pruning
process may get stuck waiting forever for a level of $\widehat{T}$ to
become defined.

If the answer to our question is positive, then look for a leaf
$\sigma$ of $T$ such that $\rng(\sigma) \subset A \cap L_{\langle e,i
\rangle}$. If no such leaf exists, then either $L_{\langle e,i
\rangle}$ is finite or $v \neq w(l_{\langle e,i \rangle})$, so end the
stage and cancel the strategies for $R_{\langle e',i' \rangle}$ with
$\langle e',i' \rangle \geq \langle e,i \rangle$ as described
above. (That is, declare $R_{\langle e',i' \rangle}$ to be
unsatisfied, and reset $r_{\langle e',i' \rangle}$, $L_{\langle e',i'
\rangle}$, and $l_{\langle e',i' \rangle}$ to the current values of
$r_{\langle e,i \rangle}$, $L_{\langle e,i \rangle}$, and $l_{\langle
e,i \rangle}$, respectively.) Otherwise, there are a nonempty $G
\subseteq \rng(\sigma)$ and an $n$ such that $\Phi^{X \uhr r_i \cup
G}_e(n)\!\downarrow \neq C_i(n)$ with use bounded by the largest
element of $G$. Let $r_{\langle e,i \rangle+1}$ be the largest element
of $G$, let $L_{\langle e,i \rangle+1}=L_{\langle e,i \rangle}$, and
let $l_{\langle e,i \rangle+1}=l_{\langle e,i \rangle}$. Put every
element of $G$ into $X$.

If the answer to our question is negative, then use the relativized
form of Theorem \ref{lbt} to $0'$-effectively obtain a low path $P$ of
$T$ such that $C_j \ntleq P$ for all $j$, and a lowness index
$l_{\langle e,i \rangle+1}$ for $L_{\langle e,i \rangle+1} = X \uhr
r_i\, \cup\, \rng(P)$. If $g$ is not total, then the construction in
the proof of Theorem \ref{lbt} will still produce such an $L_{\langle
e,i \rangle+1}$ and $l_{\langle e,i \rangle+1}$, but $L_{\langle e,i
\rangle+1}$ may be finite. (Which will of course be a problem for
weaker priority requirements, but in this case the strategy for
$R_{e,i}$ will eventually be canceled, and hence $L_{\langle e,i
\rangle+1}$ will eventually be redefined.) Let $r_{\langle e,i
\rangle+1}=r_{\langle e,i \rangle}$. Search for an element of $A \cap
L_{\langle e,i \rangle+1}$ greater than $\max\{r_{\langle e',i'
\rangle} \mid \langle e',i' \rangle \leq \langle e,i \rangle\}$ not
already in $X$ and put this number into $X$. If $L_{\langle e,i
\rangle+1}$ is infinite, such a number must be found. Otherwise, such
a number may not exist, but this situation can only happen if the
approximation to $w(l_{\langle e',i' \rangle})$ at the beginning of
stage $s$ is incorrect for some $\langle e',i' \rangle \leq \langle
e,i \rangle$, in which case the strategy for $R_{e,i}$ will be
canceled, and the stage ended as described above.

In either case, if the action of the strategy for $R_{e,i}$ has not
been canceled, then declare $R_{e,i}$ to be satisfied, and for
$\langle e',i' \rangle > \langle e,i \rangle$, declare $R_{e',i'}$ to
be unsatisfied, let $r_{\langle e',i' \rangle+1}=r_{\langle e,i
\rangle+1}$, let $L_{\langle e',i' \rangle+1}=L_{\langle e,i
\rangle+1}$, and let $l_{\langle e',i' \rangle+1}=l_{\langle e,i
\rangle+1}$.
\smallskip

This completes the construction. Since every element entering $X$ at
stage $s$ is in $A$ and is greater than $s$, we have that $X$ is a
$\Delta^0_2$ subset of $A$. Furthermore, at each stage a number is
added to $X$ unless the strategy acting at that stage is canceled, so
once we show that every requirement is eventually permanently
satisfied, we will have shown that $X$ is infinite.

Assume by induction that for all $\langle e',i' \rangle < \langle e,i
\rangle$, the requirement $R_{e',i'}$ is eventually permanently
satisfied, and that $r_{\langle e,i \rangle}$, $L_{\langle e,i
\rangle}$, and $l_{\langle e,i \rangle}$ eventually reach a final
value, for which $L_{\langle e,i \rangle}$ is infinite. Let $s$ be the
least stage by which this situation obtains and the approximation to
$w(l_{\langle e,i \rangle})$ has settled to a final value $v$. Note
that at stage $s-1$, either the strategy for some $R_{\langle e',i'
\rangle}$ with $\langle e',i' \rangle < \langle e,i \rangle$ acted, or
the approximation to $w(l_{\langle e,i \rangle})$ changed, so at the
beginning of stage $s$, it must be the case that $R_{e,i}$ is the
strongest unsatisfied requirement. Thus at that stage the strategy for
$R_{e,i}$ acts, and the function $g=\Phi^{L_{\langle e,i \rangle}}_v$
it works with at that stage is in fact a witness to the
non-$L_{\langle e,i \rangle}$-hyperimmunity of $A \cap L_{\langle e,i
\rangle}$. Thus $R_{e,i}$ will become satisfied at the end of the
stage, and $r_{\langle e,i \rangle+1}$, $L_{\langle e,i \rangle+1}$,
and $l_{\langle e,i \rangle+1}$ will not be redefined after the end of
the stage.

If the tree $T$ built at stage $s$ is finite, then a leaf $\sigma$ of
$T$ is found such that $\rng(\sigma) \subset A \cap L_{\langle e,i
\rangle}$, and there are a nonempty $G \subseteq \rng(\sigma)$ and an
$n$ such that $\Phi^{X \uhr r_{\langle e,i \rangle} \cup
G}_e(n)\!\downarrow \neq C_i(n)$ with use bounded by the largest
element $r_{\langle e,i \rangle+1}$ of $G$. Since $r_{\langle e,i
\rangle+1}$ is never again redefined, $\Phi^X_e(n) \neq C_i(n)$, and
thus the requirement $R_{e,i}$ is satisfied. Furthermore, $L_{\langle
e,i \rangle+1}$ is defined to be $L_{\langle e,i \rangle}$ at this
stage, and hence is infinite.

If $T$ is infinite, then $L_{\langle e,i \rangle+1}$ is defined to
contain the range of a path of $T$, and hence is
infinite. Furthermore, $L_{\langle e,i \rangle+1}$ is never redefined,
and by the way $X$ is defined, $X \subseteq A \cap L_{\langle e,i
\rangle+1}$. There must be an $n$ such that either
$\Phi_e^Y(n)\!\uparrow$ for every $Y \subseteq L_{\langle e,i
\rangle+1}$ or there is a $Y \subseteq L_{\langle e,i \rangle+1}$ such
that $\Phi_0^Y(n)\!\downarrow \neq C_i(n)$, since otherwise we could
$L_{\langle e,i \rangle+1}$-compute $C_i(n)$ for each $n$ by searching
for a finite $F \subset L_{\langle e,i \rangle+1}$ such that
$\Phi_e^F(n)\!\downarrow$. But by the definition of $T$, this means
that there is an $n$ such that for every infinite $Y \subseteq
L_{\langle e,i \rangle+1}$, either $\Phi_e^Y(n)\!\uparrow$ or
$\Phi_e^Y(n)\!\downarrow \neq C_i(n)$. So since $X \subseteq
L_{\langle e,i \rangle+1}$, we have $\Phi^X_e \neq C_i$, and hence the
requirement $R_{e,i}$ is satisfied.
\end{proof}

Theorem \ref{incthm} gives the negative answer to Mileti's question
mentioned above.

\begin{cor}[Hirschfeldt]
Every $\Delta^0_2$ set has an incomplete infinite $\Delta^0_2$ subset
of either it or its complement. In other words, every computable
stable $2$-coloring of $[\mathbb N]^2$ has an incomplete $\Delta^0_2$
infinite homogeneous set.
\end{cor}

We can improve on this result by using the following unpublished
result due to Jockusch.

\begin{prop}[Jockusch] \label{propgen}
Let $Z$ be hyperimmune. Then there is a $1$-generic $G \tleq Z \oplus
0'$ such that $Z \subseteq G$.
\end{prop}

\begin{proof}
We build $G$ by finite extensions; that is, we define $\gamma_0 \prec
\gamma_1 \prec \cdots$ and let $G = \bigcup_i \gamma_i$. Let
$S_0,S_1,\ldots$ be an effective listing of all c.e.\ sets of finite
binary sequences.

Begin with $\gamma_0$ defined as the empty sequence. At stage $i$,
given the finite binary sequence $\gamma_i$, search for an extension
$\alpha \in S_i$ of $\gamma_i1^{f(n)}$. If one is found then let
$f(n+1)=|\alpha|$.

If $f$ is total then, since $Z$ is hyperimmune, there is an $n$ such
that the interval $[|\gamma_i|+f(n),|\gamma_i|+f(n+1))$ contains no
element of $Z$. So $Z \oplus 0'$-computably search for either such an
interval or for an $n$ such that $f(n+1)$ is undefined. In the first
case, let $\alpha$ be as above and let $\gamma_{i+1}=\alpha$. In the
second case, let $\gamma_{i+1}=\gamma_i1^{f(n)}$.

It is now easy to check by induction that $Z \subseteq G$ and that $G$
meets or avoids each $S_i$.
\end{proof}

\begin{cor} \label{gencor}
Let $X \subset Y$ be such that $X$ is $Y$-hyperimmune. Then there are
$G,H \tleq X \oplus Y'$ such that
\begin{enumerate}

\item $H \tleq G \oplus Y$,

\item $G$ is $1$-generic relative to $Y$,

\item $X \subseteq H \subset Y$, and

\item $Y \setminus H$ is infinite.

\end{enumerate}
\end{cor}

\begin{proof}
Let $h(0)<h(1)<\cdots$ be the elements of $Y$, and let
$Z=h^{-1}(X)$. By Proposition \ref{propgen} relativized to $Y$, there
is a $G \tleq Z \oplus Y'$ such that $G$ is $1$-generic relative to
$Y$ and $Z \subseteq G$. Let $H=h(G)$. Since $h \tleq Y$ and $h$ is
increasing, we have $H \tleq G \oplus Y$, and $X \subseteq H \subset
Y$ by the definition of $h$. Finally, $Y \setminus H =
h(\overline{G})$, and hence is infinite.
\end{proof}

\begin{cor} \label{hypcor}
Let $A$ be a $\Delta^0_2$ set such that $\overline{A}$ has no infinite
low subset, and let $L$ be low. Then $A \cap L$ is not
$L$-hyperimmune.
\end{cor}

\begin{proof}
Suppose that $A \cap L$ is $L$-hyperimmune. We can apply Corollary
\ref{gencor} to $X=A \cap L$ and $Y=L$ to obtain $G$ and $H$ as
above. Since $A \cap L$ and $L'$ are both $\Delta^0_2$, so is
$G$. Since $G$ is also $1$-generic relative to $L$, and $L$ is low, $G
\oplus L$ is low. But $H \oplus L \tleq G \oplus L$, and hence $H
\oplus L$ is low. Thus $L \setminus H$ is an infinite low subset of
$\overline{A}$, which is a contradiction.
\end{proof}

\begin{cor}[Hirschfeldt] \label{original}
Let $A$ be a $\Delta^0_2$ set such that $\overline{A}$ has no infinite
low subset. Then $A$ has an incomplete infinite $\Delta^0_2$ subset.
\end{cor}

\begin{proof}[Proof Sketch]
The proof is similar to that of Theorem \ref{incthm}. Instead of
working with the given sets $C_i$, we build a $\Delta^0_2$ set $C$
while satisfying the requirements
$$
R_e : \Phi_e^X \mbox{ total} \;\; \Rightarrow\; \exists
n\,(\Phi_e^X(n) \neq C(n)).
$$

At stage $s$, we work with the least unsatisfied requirement $R_i$. We
have a number $r_i$ and a low set $L_i$ with lowness index $l_i$, such
that $L_i$ contains $X_s \uhr r_i$. We define $\widehat{T}$ as
before. For each node $\sigma = (m_0,\ldots,m_{k-1})$, if there is a
nonempty $G \subseteq \sigma$ such that $\Phi^{X \uhr r_i \cup
G}_e(s)\!\downarrow$ with use bounded by the largest element of $G$,
then we prune $\widehat{T}$ to ensure that $\sigma$ is not extendible
to an infinite path, thus obtaining a new $L_i$-computable tree $T$.

If $T$ is finite, then we look for a leaf $\sigma$ of $T$ and a $G$ as
above, let $r_{i+1}$ be the largest element of $G$, define $C(s) \neq
\Phi^{X \uhr r_i \cup G}_i(s)$, let $L_{i+1}=L_i$, let $l_{i+1}=l_i$,
and put every element of $G$ into $X$.

If $T$ is infinite, we $0'$-effectively obtain a low path $P$ of $T$
and a lowness index $l_{i+1}$ for $L_{i+1} = X \uhr r_i \cup
\rng(P)$. We then let $r_{i+1}=r_i$ and $C(s)=0$, search for an
element of $A \cap L_{i+1}$ greater than $\max\{r_j \mid j \leq i\}$
not already in $X$, and put this number into $X$.

The further details of the construction are as before, and the
verification that it succeeds in satisfying all the requirements is
similar.
\end{proof}

\end{document}